\pdfoutput=1 %arXiv admin added this line
\documentclass{article}

\usepackage{natbib}            % you should have natbib.sty
\usepackage{graphicx}          % Include this line if your 
\usepackage[latin1]{inputenc}  % å,ä,ö
\usepackage{amsmath,amssymb,dcolumn,amsthm}
\usepackage{graphicx,color,transparent}

\usepackage{algorithm}
\usepackage{algorithmic}
\usepackage{url}
\graphicspath{{./figs/}}
\usepackage{float}
\usepackage{tikz}
\usetikzlibrary{arrows}
\tikzstyle{block} = [draw, draw=black, line width = 1pt, rectangle,
    rounded corners=4pt,
    minimum height=3em, text width=.7\columnwidth,
    inner sep = 10pt]

\usepackage{ragged2e}

  \newcommand{\Kx}[1]{K^x_{#1}}
  \newcommand{\kx}[1]{k^x_{#1}}
  \newcommand{\Kl}[1]{K^\lambda_{#1}}
  \newcommand{\kl}[1]{k^\lambda_{#1}}

  \newcommand{\Klb}[1]{ K^\lambda_{#1}}
  \newcommand{\klb}[1]{ k^\lambda_{#1}}
  \newcommand{\Hb}[1]{\mathbf{H}_{#1}}
  \newcommand{\fb}[1]{\mathbf{f}_{#1}}
  \newcommand{\cb}[1]{\mathbf{c}_{#1}}
  \newcommand{\Ae}[1]{\mathbf{A}_{#1}}
  \newcommand{\be}[1]{\mathbf{b}_{#1}}
  \newcommand{\Ge}[1]{\mathbf{G}_{#1}}
  \newcommand{\Hbh}[1]{\hat H_{#1}}
  \newcommand{\fbh}[1]{\hat f_{#1}}
  \newcommand{\cbh}[1]{\hat c_{#1}}
  \newcommand{\xh}[1]{\hat x_{#1}}
  \newcommand{\uh}[1]{\hat u_{#1}}
  \newcommand{\Ah}[1]{\hat A_{#1}}
  \newcommand{\Bh}[1]{\hat B_{#1}}
  \newcommand{\ah}[1]{\hat a_{#1}}
  \newcommand{\lamN}[1]{\lambda_{#1}^{\mathcal{N}}}
  \newcommand{\thb}[1]{ \theta_{#1}}

  \newcommand{\Nb}[1]{ Z_{#1}}
  \newcommand{\lamS}[2]{\lambda_{{#1},{#2}}}
  \newcommand{\lamH}[1]{\hat \lambda_{#1}}
  \newcommand{\lamB}[2]{\gamma_{{#1},{#2}}}
  \newcommand{\Ac}[1]{\hat{A}_{#1}}
  \newcommand{\Tc}[1]{\hat{B}_{#1}}
  
  \newcommand{\ac}[1]{\hat{a}_{#1}}
  \newcommand{\MPCsub}[2]{\mathcal{P}_{#1}^{#2}(N_{#1}^{#2})}
  \newcommand{\MPC}[1]{\mathcal{P}(#1)}
  \newcommand{\lamtc}[1]{\lambda_{tc,#1}}
  \newcommand{\xbar}[1]{\hat x_{#1}}
  \newcommand{\dbar}[1]{\hat u_{#1}}
  \newcommand{\Vhat}[1]{\hat V_{#1}}

\usepackage{axearticle,axecontrol,axemath,axeopt,axeabrv,axeprog,axeIEEEarticle}
\usepackage{ian_misc,ian_math}

\begin{document}

\begin{center}
\huge {An $\Ordo{\log N}$ Parallel Algorithm for \\Newton Step Computation in\\Model Predictive Control} \\ \vspace{2mm}
\large Isak Nielsen, Daniel Axehill \\ \normalsize (Division of Automatic Control, Link\"oping University, Sweden (e-mail: \{isak.nielsen@liu.se, daniel@isy.liu.se\})
\end{center}

\textbf{Abstract}
The use of Model Predictive Control in industry is steadily increasing as more complicated problems can be addressed. Due to that online optimization is usually performed, the main bottleneck with Model Predictive Control is the relatively high computational complexity. Hence, a lot of research has been performed to find efficient algorithms that solve the optimization problem. As parallelism is becoming more commonly used in hardware, the demand for efficient parallel solvers for Model Predictive Control has increased. In this paper, a tailored parallel algorithm that can adopt different levels of parallelism for solving the Newton step is presented. With sufficiently many processing units, it is capable of reducing the computational growth to logarithmic growth in the prediction horizon. Since the Newton step computation is where most computational effort is spent in both interior-point and active-set solvers, this new algorithm can significantly reduce the computational complexity of highly relevant solvers for Model Predictive Control. 

\vspace{4mm}


%%%%%%%%%%%%%%%%%%%%%%%%%%%%%%%%%%%%%%%%%%%%%%%%%%%%%%%%%%%%%%%%%%%%%%%%%%%%
%%%%%%%%%%%%%%%%%%%%      INTRODUCTION      %%%%%%%%%%%%%%%%%%%%%%%%%%%%%
%%%%%%%%%%%%%%%%%%%%%%%%%%%%%%%%%%%%%%%%%%%%%%%%%%%%%%%%%%%%%%%%%%%%%%%%%%

\section{Introduction}
\label{sec:intro}
Model Predictive Control (MPC) is one of the most commonly used
control strategies in industry. Some important reasons for its success
include that it can handle multi-variable systems and constraints on
control signals and state variables in a structured way~\cite{maciejowski2002predictive}. In each
sample an optimization problem is solved and in the methods
considered in this paper, the optimization problem is assumed to be
solved on-line. Note that, however, similar linear algebra is also useful
off-line in explicit MPC solvers. Depending on which type of system
and problem formulation that is used the optimization
problem can be of different types, and the most common variants are
linear MPC, nonlinear MPC and hybrid MPC. In most cases, the effort
spent in the optimization problems boils down to solving
Newton-system-like equations. Hence, lots of research has been done in
the area of solving this type of system of equations efficiently when
it has the special form from MPC, see \eg\cite{jonson83:thesis,rao98:_applic_inter_point_method_model_predic_contr,hansson00:_primal_dual_inter_point_method,bartlett02:_quadr,vandenberghe02:_robus_full,akerblad04:_effic,axehill06:_mixed_integ_dual_quadr_progr_tail_mpc,axevanhan07:_relax_applic_mipc_compa_and_eff_compu,axehill08:thesis,axehill08:_dual_gradien_projec_quadr_progr,diehl09:_nonlin_model_predic_contr,nielsen13low_rank_updates}. 

In recent years, much effort has been spent on efficient parallel solutions~\cite{constantinides2009tutorial}. In \cite{soudbakhsh2013parallelized} an extended Parallel Cyclic Reduction algorithm is used to reduce the computation to smaller systems of equations that are solved in parallel. The computational complexity of this algorithm is reported to be $\Ordo{\log N}$, where $N$ is the prediction horizon. \cite{laird2011parallel}, \cite{ZhuParallelNP} and \cite{reuterswardLic} adopt a time-splitting approach to split the prediction horizon into blocks. The subproblems in the blocks are connected through common variables and are solved in parallel using Schur complements. The common variables are decided via a consensus step where a dense system of equations involving all common variables has to be solved sequentially.
In~\cite{o2012splitting} a splitting method based on Alternating Direction Method of Multipliers (ADMM) is used, where some steps of the algorithm can be computed in parallel. \cite{stathopoulos2013hierarchical} develop an iterative three-set splitting QP solver. In this method the prediction horizon is split into smaller subproblems that are in turn split into three simpler problems. All these can be computed in parallel and a consensus step using ADMM is performed to achieve the final solution. 

In this paper there are two main contributions. First, it is shown that an equality constrained MPC problem of prediction horizon $N$ can be reduced to a new, smaller MPC problem on the same form but with prediction horizon $\tilde N < N$ in parallel. Since the new problem also has the structure of an MPC problem, it can be solved in $\Ordo{N}$.
Second, by repeating the reduction procedure it can be shown that an equality constrained MPC problem corresponding to the Newton step can be solved non-iteratively in parallel, giving a computational complexity growth as low as $\Ordo{\log N}$. The major computational effort when solving an MPC problem is often spent on computing the Newton step, and doing this in parallel as proposed in this paper significantly reduces the overall computational effort of the solver.

In this article, $\posdefmats^n$ ($\possemidefmats^n$) denotes
symmetric positive (semi) definite matrices with $n$ columns. Furthermore, let $\intnums$
be the set of integers, and $\intset{i}{j} = \braces{i,i+1,\hdots,j}$. Symbols in sans-serif font (\eg $\timestack{x}$) denote vectors of stacked element.

\begin{definition}
\label{def:LICQ}
For a set of linear constraints $A x = b$, the linear independence constraint qualification (LICQ) holds if the constraint gradients are linearly independent, \ie if $A$ has full row rank. When LICQ is violated it is referred to as primal degeneracy. 
\end{definition}

%%%%%%%%%%%%%%%%%%%%%%%%%%%%%%%%%%%%%%%%%%%%%%%%%%%%%%%%%%%%%%%%%%%%%%%%%%%%%%
%%%%%%%%%%%%%%%%%%%%       PROBLEM FORMULATION     %%%%%%%%%%%%%%%%%%%%%%%%%%%
%%%%%%%%%%%%%%%%%%%%%%%%%%%%%%%%%%%%%%%%%%%%%%%%%%%%%%%%%%%%%%%%%%%%%%%%%%%%%

\section{Problem Formulation}
\label{sec:prob_form}
The optimization problem that is solved at each sample in linear MPC is a convex QP problem in the form

\begin{equation}
  \label{eq:min_problem}
 \minimize{
    &\sum^{N-1}_{t=0}\big(\frac{1}{2}\begin{bmatrix}
    x_t^T & u_t^T
    \end{bmatrix} H_t\begin{bmatrix}
    x_t \\  u_t
\end{bmatrix} +  f_t^T \begin{bmatrix}
x_t \\  u_t
\end{bmatrix}  +  c_t   \big) \\
     &+\frac{1}{2}x^T_NH_{N}x_N + f_N^T x_N + c_N}
  {\timestack{x},\timestack{u}}
  {&x_0 = \bar x \\
    &x_{t+1} = A_tx_t + B_t u_t +  a_t, \; t \in \intset{0}{N-1} \\
    & u_t \in \mathcal{U}_t, \; t \in \intset{0}{N-1} \\
    &x_t \in \mathcal{X}_t, \; t \in \intset{0}{N}}
\end{equation}
where the equality constraints are the dynamics equations of the system, and $\mathcal{U}_t$ and $\mathcal{X}_t$ are the sets of feasible control signals and states, respectively. In this paper, let the following assumptions hold for all $t$
\begin{assumption}
 $\mathcal{X}_t = \mathbb{R}^{n_x}$ and $\mathcal{U}_t$ consists of constraints of the form $ u_{t,\min} \leq  u_t \leq  u_{t,\max}$, \ie upper and lower bounds on the control signal.
\end{assumption}
\begin{assumption}
\begin{equation}
H_t = \begin{bmatrix}
H_{x,t} & H_{xu,t} \\ H_{xu,t}^T & H_{u,t}
\end{bmatrix} \in \possemidefmats^{n_x+n_u}, \; H_{u,t} \in \posdefmats^{n_u}, \; H_N \in \possemidefmats^{n_x}
\end{equation}
\end{assumption}
\begin{assumption}
The dynamical system in~\eqref{eq:min_problem} is stable.
\end{assumption}

The problem~\eqref{eq:min_problem} can be solved using different methods, see \eg~\cite{nocedal06:num_opt}. Two common methods are interior-point (IP) methods and active-set (AS) methods. IP methods approximate the inequality constraints with barrier functions, whereas the AS methods iteratively changes the set of inequality constraints that hold with equality until the optimal active set has been found. In both types, the main computational effort is spent while solving Newton-system-like equations often corresponding to an equality constrained MPC problem with prediction horizon $N$ (or to a problem with similar structure)
\begin{equation}
 \label{eq:org_eqc_problem}
 \MPC{N}: \minimize{
    &\sum^{N-1}_{t=0}\big(\frac{1}{2}\begin{bmatrix}
    x_t^T & u_t^T
    \end{bmatrix} H_t\begin{bmatrix}
    x_t \\ u_t
\end{bmatrix} +  f_t^T \begin{bmatrix}
x_t \\ u_t
\end{bmatrix}  + c_t   \big) \\
     &+\frac{1}{2}x^T_NH_{N}x_N + f_N^T x_N + c_N}
  {\timestack{x},\timestack{u}}
  {&x_0 = \bar x \\
    &x_{t+1} = A_tx_t + B_tu_t + a_t, \; t \in \intset{0}{N-1}.}
\end{equation}
Even though this problem might look simple and irrelevant it is the workhorse of many optimization routines for linear, nonlinear and hybrid MPC.
$\MPC{N}$ is the resulting problem after the equality constraints corresponding to active control signal constraints have been eliminated as in an AS method (only control signal constraints are considered). Note that $u_t$ and the corresponding matrices have potentially changed dimensions from~\eqref{eq:min_problem}. Further, let the following assumption hold
\begin{assumption}
LICQ holds for~\eqref{eq:org_eqc_problem}. \label{assum:lin_indep}
\end{assumption}

\section{Problem decomposition}
\label{sec:time_split}
The equality constrained MPC problem~\eqref{eq:org_eqc_problem} is highly structured and this could be used to split the MPC problem into smaller subproblems that only share a small number of common variables. Given the value of the common variables, the subproblems can be solved individually. These smaller subproblems are obtained by splitting the prediction horizon in $p+1$ intervals $i=0,\ldots,p$ (each of length $N_i$) and introducing initial and terminal constraints $x_{0,i} = \xh{i}$ and $x_{N_i,i} = d_i$ for each subproblem. The connection between the subproblems $i=0,\ldots,p$ are given by the coupling constraints $\xbar{i+1} = d_i $. Let $x_{t,i}$ and $u_{t,i}$ denote the state and control signal in subproblem $i$ and let the indices of the matrices be defined analogously. For notational aspects, and without loss of generality, the terminal state $d_i$ is generalized to $d_i = \Ac{i} \xbar{i} + \Tc{i} \dbar{i} + \ac{i}$, where $\xbar{i}$ and $\dbar{i}$ are the common variables. The choice of this notation will soon become clear. Then, the MPC problem~\eqref{eq:org_eqc_problem} can be cast in the equivalent form
\begin{equation}
\label{eq:org_eqc_problem_expanded}
\minimize{\sum_{i=0}^p &\sum_{t=0}^{N_{i}-1} \big(\frac{1}{2}\begin{bmatrix}
    x_{t,i}^T & u_{t,i}^T
    \end{bmatrix}H_{t,i}\begin{bmatrix}
    x_{t,i} \\ u_{t,i}
\end{bmatrix} + f_{t,i}^T \begin{bmatrix}
x_{t,i} \\ u_{t,i}
\end{bmatrix} + c_{t,i}   \big)\\&+\frac{1}{2}x_{N_{p},p}^TH_{N_p,p}x_{N_{p},p}+f_{N_{p},p}^T x_{N_{p},p} + c_{N_{p},p}}{\timestack{x},\timestack{u}}{&\xbar{0} = \bar x \\ &\forall \; i \in \intset{0}{p} \begin{cases}
x_{0,i} = \xbar{i} \\
x_{t+1,i} = A_{t,i}x_{t,i} + B_{t,i}u_{t,i}+  a_{t,i}, \\ \hspace{12em} t \in \intset{0}{N_{i}-1} \\
x_{N_{i},i} = d_{i}= \Ac{i} \xbar{i} + \Tc{i} \dbar{i} +  \ac{i}, \; i \neq p \end{cases}\\
%x_{,p} &= \xbar{p} \\
%x_{t+1,p} &= A_{t,p}x_{t,p}+B_{t,p}u_{t,p} + a_{t,p}, \; t \in \intset{0}{N_{p}-1}\\
&\xbar{i+1} = d_i = \Ac{i} \xbar{i} + \Tc{i} \dbar{i} + \ac{i}, \; i \in \intset{0}{p-1},}
\end{equation}
Note that the first initial state $\xbar{0}$ is equal to the initial state of the original problem~\eqref{eq:org_eqc_problem}. For $i = 0,\ldots,p-1$ the individual subproblems in~\eqref{eq:org_eqc_problem_expanded} are given by
%\begin{equation}
%\begin{split}
%x_{0,i} &= \xbar{i} \\
%x_{t+1,i} &= A_{t,i}x_{t,i} + B_{t,i}u_{t,i}+a_{t,i}, \; t \in \intset{0}{N_i-1} \\
%x_{N_i,i} &= \Ac{i} \xbar{i} + \Tc{i} \dbar{i} + \ac{i}
%\end{split}
%\label{eq:subproblem_dyn}
%\end{equation}
%and 
%\begin{equation}
%V_i = \frac{1}{2}\sum_{t=0}^{N_{i}-1} \big(\begin{bmatrix}
%    x_{t,i}^T & u_{t,i}^T
%    \end{bmatrix}H_{t,i}\begin{bmatrix}
%    x_{t,i} \\ u_{t,i}
%\end{bmatrix} + 2f_{t,i}^T \begin{bmatrix}
%x_{t,i} \\ u_{t,i}
%\end{bmatrix}  + 2c_{t,i}   \big)
%\label{eq:subproblem_objF}
%\end{equation}
\begin{equation}
\minimize{&\sum_{t=0}^{N_{i}-1} \big(\frac{1}{2}\begin{bmatrix}
    x_{t,i}^T & u_{t,i}^T
    \end{bmatrix}H_{t,i}\begin{bmatrix}
    x_{t,i} \\ u_{t,i}
\end{bmatrix} +f_{t,i}^T \begin{bmatrix}
x_{t,i} \\ u_{t,i}
\end{bmatrix}  + c_{t,i}   \big)}{\timestack{x},\timestack{u}}{x_{0,i} &= \xbar{i} \\
x_{t+1,i} &= A_{t,i}x_{t,i} + B_{t,i}u_{t,i}+a_{t,i}, \; t \in \intset{0}{N_{i}-1} \\
x_{N_{i},i} &= \Ac{i} \xbar{i} + \Tc{i} \dbar{i} + \ac{i}}
\label{eq:subproblem}
\end{equation}
Here $i$ is the index of the subproblem. The last problem $p$ does not have a terminal constraint and is hence only dependent on one common variable,
%\begin{equation}
%\begin{split}
%x_{0,p} &= \xbar{p} \\
%x_{t+1,p} &= A_{t,p}x_{t,p}+B_{t,p}u_{t,p} + a_{t,p}, \; t \in \intset{0}{N_p-1}.
%\end{split}
%\label{eq:subproblem_last_dyn}
%\end{equation}
%The objective function is
%\begin{equation}
%\begin{split}
%V_{p} = &\frac{1}{2}\sum_{t=0}^{N_{p}-1} \big(\begin{bmatrix}
%    x_{t,p}^T & u_{t,p}^T
%    \end{bmatrix}H_{t,p}\begin{bmatrix}
%    x_{t,p} \\ u_{t,p}
%\end{bmatrix} + 2f_{t,p}^T \begin{bmatrix}
%x_{t,p} \\ u_{t,p}
%\end{bmatrix}  + 2c_{t,p}   \big)\\
%&+\frac{1}{2}x_{N_{p},p}^TH_{N_p,p}x_{N_{p},p}+f_{N_{p},p}^T x_{N_{p},p} + c_{N_{p},p}.
%\end{split}
%\label{eq:subproblem_last_objF}
%\end{equation}
\begin{equation}
\minimize{&\sum_{t=0}^{N_{p}-1} \big(\frac{1}{2}\begin{bmatrix}
    x_{t,p}^T & u_{t,p}^T
    \end{bmatrix}H_{t,p}\begin{bmatrix}
    x_{t,p} \\ u_{t,p}
\end{bmatrix} +f_{t,p}^T \begin{bmatrix}
x_{t,p} \\ u_{t,p}
\end{bmatrix}  + c_{t,p}   \big) \\ &+\frac{1}{2}x_{N_{p},p}^TH_{N_p,p}x_{N_{p},p}+f_{N_{p},p}^T x_{N_{p},p} + c_{N_{p},p}}{\timestack{x},\timestack{u}}{x_{0,p} &= \xbar{p} \\
x_{t+1,p} &= A_{t,p}x_{t,p} + B_{t,p}u_{t,p}+a_{t,p}, \; t \in \intset{0}{N_{p}-1}.}
\label{eq:subproblem_last}
\end{equation} 
\begin{remark}
The sizes of the subproblems, \ie the values of $N_i$, do not necessarily have to be the same, allowing different sizes of the subproblems.
\end{remark}

Temporarily excluding details, each subproblem~\eqref{eq:subproblem} and~\eqref{eq:subproblem_last} can be solved parametrically and the solution to each subproblem is a function of the common variables $\xbar{i}$ and $\dbar{i}$. By inserting these parametric solutions of all subproblems in~\eqref{eq:org_eqc_problem_expanded} and using the coupling constraints between the subproblems, problem~\eqref{eq:org_eqc_problem_expanded} can be reduced to an equivalent master problem
\begin{equation}
\label{eq:red_mpc}
\MPC{p}: \minimize{&\sum_{i=0}^{p-1}\big(\frac{1}{2}\begin{bmatrix}
\xh{i}^T & \uh{i}^T
\end{bmatrix} \Hbh{i} \begin{bmatrix}
\xh{i} \\ \uh{i}
\end{bmatrix} + \fbh{i}^T\begin{bmatrix}
\xh{i} \\ \uh{i}
\end{bmatrix} + \cbh{i} \big) \\
&+\frac{1}{2}\xh{p}^T \Hbh{p}\xh{p}+\fbh{p}^T\xh{p}+\cbh{p} }
{\timestack{\xh},\timestack{\uh}}
{&\xh{0}= \bar x \\ &\xh{i+1}=\Ah{i}\xh{i}+\Bh{i}\uh{i}+\ah{i}, \; i\in \intset{0}{p-1}.}
\end{equation}
Here $\Hbh{i}$, $\fbh{i}$ and $\cbh{i}$ are computed in each subproblem and represents the value function. The dynamics constraints in the master problem are given by the coupling constraints between the subproblems. This new MPC problem is on the same form as the original equality constrained problem~\eqref{eq:org_eqc_problem}, but with prediction horizon $p < N$. The reduction of the problem is summarized in Theorem~\ref{thm:reduce_mpc} and is graphically depicted in Fig.~\ref{fig:mpc_red_struct}, where the dotted lines represents repetition of the structure.  This approach is similar to primal decomposition~\cite{Lasdon1970optimization},~\cite{primaldecomp} where the $p+1$ subproblems  share common variables $\xbar{i}$ and $\dbar{i}$ that are computed iteratively. In the work presented in this paper the common variables are however not computed iteratively but instead determined by solving the new, reduced MPC problem at the upper level in Fig.~\ref{fig:mpc_red_struct}. Inserting the optimal $\xbar{i}$ and $\dbar{i}$ into the subproblems given by~\eqref{eq:subproblem} and~\eqref{eq:subproblem_last} gives the solution to~\eqref{eq:org_eqc_problem}. 
\begin{figure}
\centering
\def\svgwidth{0.85\columnwidth}
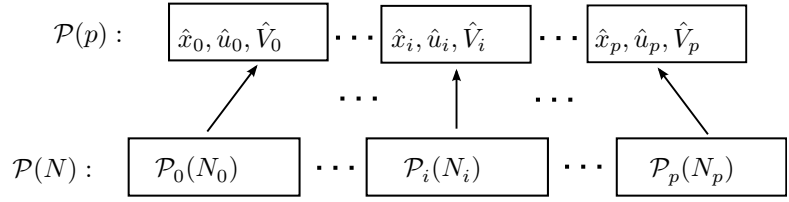
\caption{The parameters $\xh{i}$ and $\uh{i}$ in each subproblem $0,\ldots,i,\ldots,p$ can be interpreted as new state and control variables in the reduced MPC problem with prediction horizon $p$. The value functions $\Vhat{i}(\xh{i},\uh{i})$ are the terms in the new objective function.}
\label{fig:mpc_red_struct}
\end{figure}
\begin{theorem}
Consider an optimization problem $\MPC{N}$ defined in~\eqref{eq:org_eqc_problem} where Assumption~\ref{assum:lin_indep} holds.
Then $\MPC{N}$ can be reduced to $\MPC{p}$ in parallel, where $1 \leq p < N$. The optimal solution $X^*$ and $\lambda^*$ to $\MPC{N}$ can be computed in parallel from the solution $\hat X^*$ and $\hat \lambda^*$ to $\MPC{p}$. \label{thm:reduce_mpc}
\end{theorem}
\begin{proof}
For the proof of Theorem~\ref{thm:reduce_mpc}, see Appendix~\ref{app:subsec:proof_lemma_reduction}. 
\end{proof}
In the rest of this section it will be shown how the subproblems~\eqref{eq:subproblem} and~\eqref{eq:subproblem_last} are solved parametrically and how the matrices needed in~\eqref{eq:red_mpc} are computed.

%%%%%%%%%%%%%%%     FIRST BLOCKS           %%%%%%%%%%

\subsection{Solution of the first subproblems $i=0,...,p-1$}
\label{subsec:mp_first_blocks}
In this section, it will be shown that each subproblem $i=0,\ldots,p-1$ given by~\eqref{eq:subproblem} can be solved parametrically and that the solution can be expressed as a function of the common variables $\xh{i}$ and $\uh{i}$.
%\begin{equation}
%\minimize{&\frac{1}{2}\sum_{t=0}^{N_{i}-1} \big(\begin{bmatrix}
%    x_{t,i}^T & u_{t,i}^T
%    \end{bmatrix}H_{t,i}\begin{bmatrix}
%    x_{t,i} \\ u_{t,i}
%\end{bmatrix} +2f_{t,i}^T \begin{bmatrix}
%x_{t,i} \\ u_{t,i}
%\end{bmatrix}  + 2c_{t,i}   \big)}{\timestack{x},\timestack{u}}{x_{0,i} &= \xbar{i} \\
%x_{t+1,i} &= A_{t,i}x_{t,i} + B_{t,i}u_{t,i}+a_{t,i}, \; t \in \intset{0}{N_{i}-1} \\
%x_{N_{i},i} &= \Ac{i} \xbar{i} + \Tc{i} \dbar{i} + \ac{i}}
%\label{eq:subproblem}
%\end{equation}
For now it is assumed that LICQ holds for~\eqref{eq:subproblem}. The optimization problem can be cast in the more compact form
\begin{equation}
\minimize{\frac{1}{2} X_i^T \Hb{i} X_i + \fb{i}^TX_i+\cb{i}}{X_i}{\Ae{i}X_i = \be{i} + \Ge{i} \theta_i,}
\label{eq:mp_subproblem} 
\end{equation}
by defining
\begin{equation}
X_i \triangleq \begin{bmatrix}
x_{0,i} \\ u_{0,i} \\ \vdots \\ u_{N_{i}-1,i} \\ x_{N_{i},i} 
\end{bmatrix} \label{eq:def_Xi}, \quad \lambda_i \triangleq \begin{bmatrix}
\lamS{0}{i} \\ \lamS{1}{i} \\ \vdots \\ \lamS{N_i}{i} \\ \lamtc{i}
\end{bmatrix}, \quad \cb{i} \triangleq \sum_{t = 0}^{N_{i}-1} c_{t,i} ,
\end{equation} 
\begin{equation}
\Hb{i}  \triangleq \begin{bmatrix}
H_{0,i} & 0 & \cdots &  0\\
0 & \ddots & \ddots &  \vdots \\
%\vdots & \ddots & & \ddots \\
\vdots & \ddots & H_{{N_i-1},i} & 0\\
0 & \cdots & 0 & 0
\end{bmatrix}, \quad
\fb{i} \triangleq \begin{bmatrix}
f_{0,i} \\ \vdots \\ f_{N_{i}-1,i} \\ 0
\end{bmatrix}, 
\end{equation}
\begin{equation}
\Ae{i} \triangleq \begin{bmatrix}
-I & 0 & \cdots & & & \cdots & 0\\
A_{0,i} & B_{0,i} & -I & 0 & \cdots & & \vdots\\
0 & 0 &A_{1,i} & \cdots \\
\vdots & \vdots & & \ddots & & & 0 \\
 \vdots & \vdots & &              & A_{N_i-1,i} & B_{N_{i}-1,i} & -I \\
0 &   \cdots  & &   &    \cdots &   0          & -I \\
\end{bmatrix},
\end{equation}
\begin{equation}
\be{i} \triangleq \begin{bmatrix}
0 \\ -a_{0,i} \\ \vdots \\ -a_{N_i-1,i} \\ - \ac{i}
\end{bmatrix}, \quad
\Ge{i} \triangleq \begin{bmatrix}
-I & 0 \\ 0 & 0 \\ \vdots &  \vdots \\ 0 & 0 \\ -\Ac{i} & -\Tc{i}
\end{bmatrix}, \quad
\theta_i \triangleq \begin{bmatrix}
\xbar{i} \\ \dbar{i}
\end{bmatrix}, 
\label{eq:def_Gi}
\end{equation}

The dual variables $\lambda_i$ in the subproblem are introduced as 
\begin{align}
\lamS{0}{i} \leftrightarrow x_{0,i} &= \xbar{i} \label{eq:lam0_dual_vars_ctrl}\\
\lamS{t+1}{i} \leftrightarrow x_{t+1,i}  &= A_{t,i}x_{t,i} + B_{t,i} u_{t,i} + a_{t,i} , \; t \in \intset{0}{N_i-1} \label{eq:lamt_dual_vars_ctrl} \\
\lamtc{i} \leftrightarrow x_{N_i,i}  &=  \Ac{i} \xbar{i} + \Tc{i} \dbar{i} +  \ac{i}. \label{eq:lamd_dual_vars_ctrl}
\end{align}
The symbol $\leftrightarrow$ should be interpreted as $\lambda$ being the dual variable corresponding to the respective equality constraint. 

Note that~\eqref{eq:mp_subproblem} is a very simple multiparametric quadratic programming problem with parameters $\theta_i$ and only equality constraints. Hence the optimal primal and dual solution to this problem are both affine functions of the parameters~$\theta_i$,~\cite{TondellMPQP}.
\begin{remark}
Since the simple parametric programming problem~\eqref{eq:mp_subproblem} is subject to equality constraints only it is not \emph{piecewise} affine in the parameters. Hence, the solution can be computed cheaply and it does not suffer from the complexity issues of a general multiparametric programming problem.
\end{remark}
Since LICQ is assumed to hold, the unique optimal primal solution can be expressed as
\begin{equation}
X_i^*(\theta_i) = \Kx{i} \theta_i + \kx{i},
\label{eq:mp_primal_sol}
\end{equation}
and similarly for the unique optimal dual solution
\begin{equation}
\lambda_i^*(\theta_i) = \Kl{i} \theta_i + \kl{i},
\label{eq:mp_dual_sol}
\end{equation}
for some $\Kx{i}$, $\kx{i}$, $\Kl{i}$ and $\kl{i}$, and where $i$ denotes the index of the subproblem. 
%Partition $\Kx{i}$ and $\kx{i}$ into
%\begin{align}
%\Kx{i} &= \begin{bmatrix}
%(\Kx{i})^1_{0} & (\Kx{i})^2_{0} \\
%\vdots & \vdots \\
%(\Kx{i})^1_{N_{i}} & (\Kx{i})^2_{N_{i}} 
%\end{bmatrix} \label{eq:Kx_partition}\\
%\kx{i} &= \begin{bmatrix}
%(\kx{i})_{0} \\
%\vdots \\
%(\kx{i})_{N_{i}}
%\end{bmatrix} \label{eq:kx_partition}
%\end{align}
%with each block in~\eqref{eq:Kx_partition} and~\eqref{eq:kx_partition} of appropriate dimensions. Here the subindex $i$ refers to subproblem $i$. Then the optimal solution of the last state $x_{N_{i},i}$ in subproblem $i$ can be expressed by the parameters $\bar x_i$ and $\bar d_i$, giving
%\begin{equation}
%x_{N_{i},i}^* = (\Kx{i})_{N_{i}}^1\bar x_{i} + (\Kx{i})_{N_{i}}^2 \bar d_i + (\kx{i})_{N_{i}}.
%\label{eq:mp_last_state}
%\end{equation}
The value function of~\eqref{eq:subproblem} is obtained by inserting the parametric primal optimal solution~\eqref{eq:mp_primal_sol} into the objective function in~\eqref{eq:mp_subproblem}, with the result
\begin{equation}
\begin{split}
&\Vhat{i}(\theta_i) = \frac{1}{2}(\theta_i ^T (\Kx{i})^T + (\kx{i})^T)\Hb{i}(\Kx{i} \theta_i + \kx{i}) +  \\
%&\fb{i}^T(\Kx{i} \theta_i +\kx{i}) +\cb{i}
%= \frac{1}{2}\theta_i^T(\Kx{i})^T\Hb{i}\Kx{i} \theta_i + (\kx{i})^T \Hb{i}\Kx{i} \theta_i +\\ 
 & \fb{i}^T(\Kx{i} \theta_i + \kx{i}) + \cb{i}
= \frac{1}{2}\theta_i^T \Hbh{i} \theta_i + \fbh{i}^T\theta_i + \cbh{i}
\end{split}
\label{eq:mp_subproblem_sol_objF}
\end{equation}
where $\Hbh{i} = (\Kx{i})^T\Hb{i}\Kx{i}$, $\fbh{i}^T = \fb{i}^T\Kx{i}+(\kx{i})^T\Hb{i}\Kx{i}$ and $\cbh{i} = \cb{i}+\frac{1}{2}(\kx{i})^T\Hb{i}\kx{i}+\fb{i}^T\kx{i}$.

%%%%%%%%%%%       LAST BLOCK    %%%%%%%%%%%%%%%%%%

\subsection{Solution of the last subproblem $i=p$}
\label{subsec:mp_last_block}
The last subproblem~\eqref{eq:subproblem_last} is different from the $p$ first since there is no terminal constraint on $x_{N_{p},p}$. Hence the parametric solution of this problem only depends on the initial state $\xbar{p}$ of the subproblem. The derivation of the solution is analogous to the one in Section~\ref{subsec:mp_first_blocks}, but with $\theta_p = \xbar{p}$. The unique optimal primal solution to 
\begin{equation}
\minimize{\frac{1}{2} X_{p}^T \Hb{p} X_{p} + \fb{p}^TX_{p}+\cb{p}}{X_{p}}{\Ae{p}X_{p} = \be{p} + \Ge{p} \theta_{p},}
\label{eq:mp_subproblem_last} 
\end{equation}
is given as the affine function
\begin{equation}
X_{p}^* = \Kx{p}\theta_{p}+\kx{p} = \Kx{p}\xbar{p} + \kx{p},
\label{eq:mp_last_primal_sol}
\end{equation}
and the unique optimal dual solution is 
\begin{equation}
\lambda_{p}^* = \Kl{p}\theta_{p}+\kl{p} = \Kl{p}\xbar{p} + \kl{p}.
\label{eq:mp_last_dual_sol}
\end{equation}
The dual variables $\lambda_p$ are defined as in~\eqref{eq:def_Xi}, but the last dual variable $\lamtc{i}$ corresponding to the terminal constraint does not exist. The same notation as in Section~\ref{subsec:mp_first_blocks} has been used, with the slight difference that the last blocks in $\Hb{p}$ and $\fb{p}$ are $H_{N_p,p}$ and $f_{N_p,p}$ respectively. Furthermore, the sum when computing $\cb{p}$ is also including $t=N_p$, the last block rows in $\Ae{i}$, $\be{i}$ and $\Ge{i}$ are removed and the last column of $\Ge{i}$ is removed (all corresponding to the constraint and parameter that is not present in the last subproblem).

%Since there is only one parameter, $\Kx{p}$ is partitioned as
%\begin{equation}
%\Kx{p} = \begin{bmatrix}
%(\Kx{p})_{0} \\
%\vdots \\
%(\Kx{p})_{N_{p}}
%\end{bmatrix}
%\end{equation}
%and the optimal final state $x_{N_p,p}$ is expressed as
%\begin{equation}
%x_{N_{p},p}^* = (\Kx{p})_{N_{p}}\bar x_{p} + (\kx{p})_{N_{p}} 
%\end{equation}

Inserting the solution~\eqref{eq:mp_last_primal_sol} into the objective function of subproblem~\eqref{eq:mp_subproblem_last} gives the value function $\Vhat{p}(\theta_{p})$ as
\begin{equation}
\begin{split}
\Vhat{p}(\theta_{p}) &= \frac{1}{2}\theta_{p}^T \Hbh{p} \theta_{p} + \fbh{p}^T\theta_{p} + \cbh{p}, \\
\end{split}
\label{eq:mp_subproblem_last_sol_objF}
\end{equation}
where $\Hbh{p}$, $\fbh{p}$ and $\cbh{p}$ are defined as before.

%%%%%%%%%%%%%%%       PRIMAL DEGENERATE SUBPROBLEM     %%%%%%%%%%%%%%%%%%%%

\subsection{Solution of a primal degenerate subproblem}
\label{subsec:mp_overdetermined}
The terminal constraint in a subproblem given by~\eqref{eq:mp_subproblem} introduces $n_x$ new constraints, which might result in an infeasible subproblem or that LICQ is violated for the subproblem even though this is not the case in the original problem~\eqref{eq:org_eqc_problem}. According to Definition~\ref{def:LICQ}, violation of LICQ is known as primal degeneracy and the dual variables for a primal degenerate problem are non-unique,~\cite{TondellMPQP}. In this section it will be shown how to choose the parameter in the terminal constraint to achieve a feasible problem and also how to choose the dual variables of subproblem $i$ to coincide with the corresponding dual solution to the original problem~\eqref{eq:org_eqc_problem}.

Since the subproblem is feasible only if there exists a solution to $\Ae{i}X_i = \be{i} + \Ge{i} \theta_i$ it is required that $\be{i}+\Ge{i} \theta_i \in \range{\Ae{i}}$.
This is satisfied if the terminal constraint is chosen carefully, which means that it has to be known which $\theta_i$ that will give a feasible solution. To do this, the dynamics constraints in subproblem $i$ can be used to compute the final state in subproblem $i$ given the control signals $\timestack{u}_i$ and the initial state $\xbar{i}$ as
\begin{equation}
\begin{split}
&x_{N_{i},i} = \underbrace{\prod_{t=0}^{N_{i}-1}A_{t,i}}_{\triangleq \mathcal{A}_i} \xbar{i} + \underbrace{\begin{bmatrix}
\prod_{t=1}^{N_{i}-1} A_{t,i} & \cdots & A_{N_{i}-1,i} & I
\end{bmatrix}}_{\triangleq \mathcal{D}_i}\timestack{a}_i+ \\ &\underbrace{\begin{bmatrix}
\prod_{t=1}^{N_{i}-1}A_{t,i}B_{0,i} & \cdots & A_{N_{i}-1,i}B_{N_{i}-2,i}  & B_{N_{i}-1,i}
\end{bmatrix}}_{\triangleq \mathcal{S}_i}\timestack{u}_i \Rightarrow \\
&x_{N_{i},i} = \mathcal{A}_i\xbar{i} + \mathcal{S}_i\timestack{u}_i + \mathcal{D}_i \timestack{a}_i,
\end{split} \label{eq:definiton_A_S_D}
\end{equation}
where $\mathcal{S}_i$ can be recognized as the controllability matrix,
\begin{equation}
\timestack{a}_i = \begin{bmatrix}
a_{0,i} \\ \vdots \\ a_{N_{i}-1,i}
\end{bmatrix}, \quad \timestack{u}_i = \begin{bmatrix}
u_{0,i} \\ \vdots \\ u_{N_{i}-1,i}
\end{bmatrix},
\end{equation}
and
\begin{equation}
\prod_{t=t_0}^{t_1} A_t = A_{t_1}\cdots A_{t_0}.
\end{equation}
The feasibility of the subproblem can be ensured by a careful selection of the parametrization of the problem. In this work this is performed by requiring that the final state satisfies the terminal constraint $x_{N_{i},i} = d_i =  \Ac{i} \xbar{i} + \Tc{i} \dbar{i} + \ac{i}$, where $d_i$ is within the controllable subspace given by $\mathcal{A}_i$, $\mathcal{S}_i$ and $\mathcal{D}_i\timestack{a}_i$. This can be assured by requiring
%\begin{equation}
%d_i = \mathcal{A}_i \bar x_i + \mathcal{T}_i\bar d_i + \mathcal{D}_i\mathbf{a}_i,
%\label{eq:d_range}
%\end{equation}
\begin{equation}
\Ac{i} = \mathcal{A}_i, \quad \Tc{i} = \mathcal{T}_i,\quad \ac{i} = \mathcal{D}_i \timestack{a}_i, \label{eq:d_range}
\end{equation}
where the columns of $\mathcal{T}_i$ form a basis for the range space of $\mathcal{S}_i$. (Note that for a non-degenerate problem, $\Ac{i} = 0$, $\Tc{i} = I$ and $\ac{i}=0$ are valid choices since $\mathcal{S}_i$ has full row rank and $\mathcal{X}_t = \mathbb{R}^{n_x}$.) By using this parametrization, the master problem can only use parameters in the subproblem that will result in a feasible subproblem.
%By defining the new parameters as
%\begin{equation}
%\thb{i} \triangleq \begin{bmatrix}
%\bar x_i \\ \bar d_i
%\end{bmatrix},
%\end{equation}
%and using~\eqref{eq:d_range}, the parameters $\theta_i$ can be written
%\begin{equation}
%\theta_i = \underbrace{\begin{bmatrix}
%I & 0 \\ \mathcal{A}_i & \mathcal{T}_i
%\end{bmatrix}}_{\triangleq T_i} \thb{i} + \underbrace{\begin{bmatrix}
%0 \\ \mathcal{D}_i \mathbf{a}_i
%\end{bmatrix}}_{\triangleq D_i} = T_i \thb{i} + D_i
%\label{eq:dbar_transform}
%\end{equation}
\begin{remark}
Computation of $\mathcal{A}_i$, $\mathcal{D}_i$ and $\mathcal{S}_i$ might give numerical issues if the dynamical system in~\eqref{eq:min_problem} is unstable. So for numerical reasons, only stable systems are considered in this paper.
\end{remark}

The optimal parametric primal and dual solutions to a primal degenerate problem on the form~\eqref{eq:mp_subproblem} are given by~\eqref{eq:mp_primal_sol} and
\begin{equation}
\lambda^*_i(\theta_i) = \Kl{i}\theta_i+\kl{i}+\lamN{i},
\label{eq:mp_dual_sol_underdet}
\end{equation}
where $\lamN{i} \in \ker{\Ae{i}^T}$,~\cite{TondellMPQP}. The null space is given by Theorem~\ref{thm:nullspace}.
\begin{theorem} \label{thm:nullspace}
The null space of $\Ae{i}^T$ is given by
\begin{equation}
\ker{\Ae{i}^T} = \{ z \; | \; z= \Nb{i} w_i, \; \forall w_i \in \ker{\mathcal{S}_i^T} \},
\end{equation}
where
\begin{equation}
\Nb{i} \triangleq \begin{bmatrix}
-\Ac{i} & -\mathcal{D}_i & I
\end{bmatrix}^T, \label{eq:null_Z}
\end{equation}
and $\mathcal{S}_i$ is the controllability matrix.
\end{theorem} 
\begin{proof}
For the proof of Theorem~\ref{thm:nullspace}, see Appendix~\ref{app:nullspace}.
\end{proof}
\begin{remark}
Note that $Z_i$ is computed cheaply since the matrices $\Ac{i}$ and $\mathcal{D}_i$ are already computed.
\end{remark}
The dual variables of~\eqref{eq:org_eqc_problem_expanded} are introduced by~\eqref{eq:lam0_dual_vars_ctrl}-\eqref{eq:lamd_dual_vars_ctrl} for each subproblem, and by
\begin{align}
&\lamH{-1} \leftrightarrow \xbar{0} = \bar x \label{eq:lamh_min_one_ctrl} \\
&\lamH{i} \leftrightarrow  \xbar{i+1} =  \Ac{i} \xbar{i} + \Tc{i} \xbar{i} +\ac{i}, \; i \in \intset{0}{p-1} \label{eq:lamh_dual_vars_ctrl}
\end{align}
for the coupling constraints that connect the subproblems in~\eqref{eq:org_eqc_problem_expanded}. Note that $\lamtc{i}$ in~\eqref{eq:lamd_dual_vars_ctrl} is the dual variable corresponding to the terminal constraint in each subproblem, whereas~\eqref{eq:lamh_dual_vars_ctrl} are the dual variables corresponding to the coupling constraints between the subproblems (interpreted as the dynamics constraints in the reduced MPC problem~\eqref{eq:red_mpc}). Hence, $\lamS{tc}{i}$ is computed in the subproblem, and $\lamH{i}$ is computed when~\eqref{eq:red_mpc} is solved. This is depicted in Fig.~\ref{fig:dual_struct} where the upper level corresponds to problem~\eqref{eq:red_mpc} and the lower level to problem~\eqref{eq:org_eqc_problem}. For primal degenerate subproblems, the dual solution is non-unique.
In order to choose a dual solution to the subproblems that coincides with the original non-degenerate problem, the relations between the dual variables of different subproblems are studied. These relations are given by Theorem~\ref{thm:dual_vars_overdet} and Corollary~\ref{cor:dual_vars_underdet}.
\begin{figure}
\centering
\def\svgwidth{1\columnwidth}
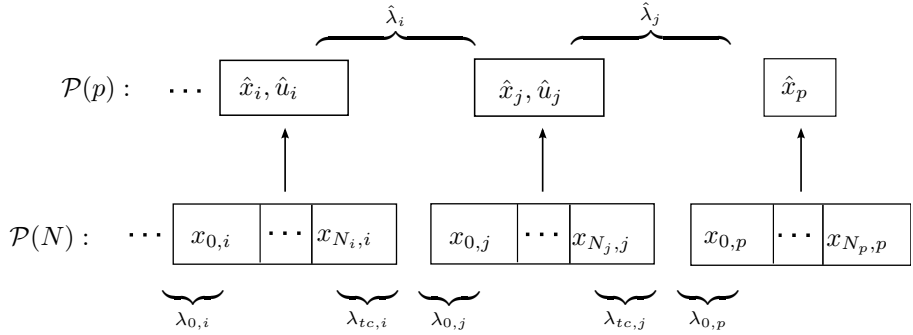
\caption{The dual variables $\lamH{i}$ in the reduced problem are connected to the dual variables in the original problem. Here $\lamtc{i}$ (the dual variables for the terminal constraint) and $\lamS{0}{i}$ (the dual variable for the initial constraint) are computed in subproblem $i$ whereas $\lamH{i}$ (the dual variable for the coupling constraint between subproblem $i$ and $i+1$) is computed when solving~\eqref{eq:red_mpc}.}
\label{fig:dual_struct}
\end{figure}
\begin{theorem} \label{thm:dual_vars_overdet}
Consider an MPC problem on the form~\eqref{eq:org_eqc_problem_expanded} where Assumption~\ref{assum:lin_indep} holds. Let the dual variables be defined by~\eqref{eq:lam0_dual_vars_ctrl},~\eqref{eq:lamt_dual_vars_ctrl},~\eqref{eq:lamd_dual_vars_ctrl},~\eqref{eq:lamh_min_one_ctrl} and~\eqref{eq:lamh_dual_vars_ctrl}. Then the relations between the optimal dual solutions in different subproblems are given by
\begin{align}
&\lamS{0}{p} = \lamH{p-1}  \label{eq:lem:lamSp_eq_lamHp}\\
&\lamS{0}{i} = \lamH{i-1} - \Ac{i}^T \left( \lamtc{i} + \lamH{i} \right), \; i \in \intset{0}{p-1} \label{eq:lem:lamH_eq_lamS} \\
&\Tc{i}^T\left( \lamtc{i}+\lamH{i} \right) = 0 , \; i \in \intset{0}{p-1}\label{eq:Tlam_eq_0} \\
&\lamS{N_i}{i} = - \lamtc{i}, \; i \in \intset{0}{p-1}, \label{eq:lem:lamN_eq_lamd}
\end{align}
where $\Ac{i}$ and $\Tc{i}$ are defined by~\eqref{eq:d_range}. 
\end{theorem}
\begin{proof}
For the proof of Theorem~\ref{thm:dual_vars_overdet}, see Appendix~\ref{app:subsec:proof_thm_overdet}.
\end{proof}
\begin{corollary} \label{cor:dual_vars_underdet}
Let the assumptions in Theorem~\ref{thm:dual_vars_overdet} be satisfied, and let LICQ hold for all subproblems $i=0,\ldots,p$. Then the optimal dual variables in the subproblems are unique and the relations between the dual solutions in the subproblems are given by
\begin{align}
&\lamS{0}{i} = \lamH{i-1} , \; i \in \intset{0}{p} \label{eq:dual_equality_1}\\
&\lamtc{i} = -\lamH{i} = - \lamS{0}{i+1}, \; i \in \intset{0}{p-1} \\
&\lamS{N_i}{i} = - \lamtc{i} = \lamS{0}{i+1}, \; i \in \intset{0}{p-1} \label{eq:dual_equality_2}.
\end{align}
\end{corollary}
\begin{proof}
Let LICQ hold for all subproblems $i \in \intset{0}{p-1}$ in~\eqref{eq:org_eqc_problem_expanded}. Then $\ker{\Ae{i}^T} = \emptyset, \; i \in \intset{0}{p-1}$ and the dual solution is unique. Furthermore, 
\begin{align}
&\textrm{rank}(\mathcal{S}_i) = n_x \Rightarrow \mathcal{T}_i=\Tc{i} \textrm{ non-singular} \Rightarrow \\ & \{\textrm{Using~\eqref{eq:Tlam_eq_0} in Theorem~\ref{thm:dual_vars_overdet}} \} \Rightarrow \lamtc{i} = - \lamH{i}. \label{eq:thm:pf:lamd_lamH}
\end{align}
Inserting~\eqref{eq:thm:pf:lamd_lamH} into~\eqref{eq:lem:lamH_eq_lamS} and~\eqref{eq:lem:lamN_eq_lamd} gives $\lamS{N_i}{i} = \lamS{0}{i+1}, \;i \in \intset{0}{p-1}$, and  $\lamS{0}{i} = \lamH{i-1}, \; i \in \intset{0}{p}$ by also using~\eqref{eq:lem:lamSp_eq_lamHp}.
\end{proof}

%\begin{lem} \label{lem:dual_vars_overdet}
%Let Assumption~\ref{assum:lin_indep} hold and $\lamH{i}$ be given. Then, for an MPC-problem on the form~\eqref{eq:org_eqc_problem_expanded} where LICQ is violated for some of the subproblems $i=0,\ldots,p$ and the dual variables are defined by~\eqref{eq:lam0_dual_vars_ctrl},~\eqref{eq:lamt_dual_vars_ctrl},~\eqref{eq:lamd_dual_vars_ctrl} and~\eqref{eq:lamh_dual_vars_ctrl}, the optimal dual variables for the subproblems are not unique. However, they can be uniquely determined to coincide with those in the original problem. Furthermore, the coupling between the optimal dual solutions in different subproblems are given by
%\begin{align}
%&\lamS{0}{p} = \lamH{p-1}  \\
%&\lamS{0}{i} = \lamH{i-1} - \Ac{i}^T \left( \lamtc{i} + \lamH{i} \right), \; i \in \intset{0}{p-1} \label{eq:lem:lamH_eq_lamS} \\
%&\Tc{i}^T\left( \lamtc{i}+\lamH{i} \right) = 0 , \; i \in \intset{0}{p-1}\label{eq:Tlam_eq_0} \\
%&\lamS{N_i}{i} = - \lamtc{i}, \; i \in \intset{0}{p-1}, \label{eq:lem:lamN_eq_lamd}
%\end{align}
%where $\Ac{i}$ and $\Tc{i}$ are defined by~\eqref{eq:d_range}. 
%
%\todo{Maybe move this to some other place}
%The nullspace of $\Ae{i}^T$ is given by
%\begin{equation}
%\ker{\Ae{i}^T} = \{ z \; | \; z= \Nb{i} w_i, \; \forall w_i \in \ker{\mathcal{S}_i^T},
%\end{equation}
%where
%\begin{equation}
%\Nb{i} \triangleq \begin{bmatrix}
%-\Ac{i} & -\mathcal{D}_i & I
%\end{bmatrix}^T,
%\end{equation}
%and $\mathcal{S}_i$ is the controllability matrix.
%\end{lem}
%\begin{pf}
%For proof of Lemma~\ref{lem:dual_vars_overdet}, see Appendix~\ref{app:dual_proof}.
%\end{pf}

Theorem~\ref{thm:nullspace} is used to choose the null space element $\lamN{i}$ in~\eqref{eq:mp_dual_sol_underdet} to obtain the correct dual solution for subproblem~$i$. According to the theorem $\lamN{i}$ can be computed as
$\lamN{i} = \Nb{i}w_i, \; w_i \in \ker{\mathcal{S}_i^T}$,
giving the optimal dual variables for subproblem $i$ as
\begin{equation}
\lambda_i^*( \theta_i,w_i) = \Klb{i}  \theta_i + \klb{i} + \Nb{i}w_i, \; w_i \in \ker{\mathcal{S}_i^T}.
\label{eq:mp_sub_dual_sol_ns_underdet}
\end{equation}
Let $ \gamma_i = \Klb{i} \theta_i + \klb{i}$ be the dual solution when the minimum norm null space element is selected, and let $\lamH{i}$ be given by the solution to problem~\eqref{eq:red_mpc}. Then it follows from Theorem~\ref{thm:dual_vars_overdet} that
\begin{align}
&\lamB{0}{i} =  \lamH{i-1} - \Ac{i}^T(\lamB{tc}{i} + \lamH{i}), \; i \in \intset{0}{p-1} \label{eq:lam0_min_lamh_plus_A_mult_lamd_plus_lamh_unshifted}\\
&\Tc{i}^T(\lamB{tc}{i} + \lamH{i}) = 0, \; i \in \intset{0}{p-1} \label{eq:T_mult_lamd_plus_lamh_unshifted} \\
&\lamB{N_i}{i}= - \lamB{tc}{i}, \; i \in \intset{0}{p-1} \label{eq:lamNi_eq_min_lamd_unshifted}.
\end{align}
To obtain the same optimal dual solution $\lambda_i$ as for the non-degenerate original problem, the freedom in the choice of the dual variables from~\eqref{eq:mp_sub_dual_sol_ns_underdet} is exploited, \ie,
\begin{equation}
\lambda_i = \gamma_i + \Nb{i}w_i.
\end{equation}
In order to obtain the relation $\lamS{N_i}{i} = \lamH{i} = \lamS{0}{i+1} $ as in the non-generate case,~\eqref{eq:lem:lamH_eq_lamS}-\eqref{eq:lem:lamN_eq_lamd} give that $\lamtc{i} = - \lamH{i}$ must hold. The last block in~\eqref{eq:null_Z} and~\eqref{eq:mp_sub_dual_sol_ns_underdet} gives $\lamtc{i} = \lamB{tc}{i} + w_i$, and based on this $w_i$ is chosen as 
\begin{equation}
w_i = -(\lamB{tc}{i}+\lamH{i}) \in \ker{\mathcal{S}_i^T} \Rightarrow \lamtc{i} = - \lamH{i} \label{eq:lamN_value}.
\end{equation} 
Note that~\eqref{eq:T_mult_lamd_plus_lamh_unshifted} gives that $w_i \in \ker{\mathcal{S}_i^T}$. By using this choice of $w_i$ in the optimal dual solution~\eqref{eq:mp_sub_dual_sol_ns_underdet} together with~\eqref{eq:null_Z},~\eqref{eq:lam0_min_lamh_plus_A_mult_lamd_plus_lamh_unshifted} and~\eqref{eq:lamNi_eq_min_lamd_unshifted} the following hold 
\begin{align}
&\lamS{0}{i} = \lamB{0}{i} + \lamN{0,i} = \lamB{0}{i}- \Ac{i}^T w_i = \lamH{i-1}, \; i \in \intset{0}{p-1} \label{eq:app:lam0_min_lamh_plus_A_mult_lamd_plus_lamh_shifted}\\
&\lamS{N_i}{i} = \lamB{N_i}{i} + \lamN{N_i,i} = - \lamB{tc}{i} -w_i = \lamH{i} \label{eq:app:lamNi_eq_min_lamd_shifted}
\end{align}
Hence, the chosen optimal dual solution of subproblem $i$ coincides with the one for the non-degenerate case if it is computed as
\begin{equation}
\lambda_i^*(\thb{i},\lamH{i}) = \Klb{i} \thb{i}+\klb{i} - \Nb{i}(\lamB{tc}{i}+\lamH{i}) = \gamma_i - \Nb{i}(\lamB{tc}{i}+\lamH{i}).  \label{eq:mp_dual_sol_corrected}
\end{equation}
The dual solution to the original problem can be retrieved from~\eqref{eq:mp_dual_sol_corrected} for $i=0,\ldots,p-1$ and~\eqref{eq:mp_last_dual_sol} for $i=p$.

%%%%%%%%%%%%%%%%%%%%%%%%%%%%%%%%%%%%%%%%%%%%%%%%%%%%%%%%%%%%%%%%%%%%%%%%%
%%%%%%%%%%%%%%%%%%      REDUCE MPC-PROBLEM     %%%%%%%%%%%%%%%%%%%%%%%%%
%%%%%%%%%%%%%%%%%%%%%%%%%%%%%%%%%%%%%%%%%%%%%%%%%%%%%%%%%%%%%%%%%%%%%%%%%

\section{Problem reduction in parallel}
\label{sec:reduced_mpc}

%It will now be shown that the original MPC-problem~\eqref{eq:org_eqc_problem} can be reduced to an MPC-problem with shorter prediction horizon by splitting the problem into several smaller problems, and how this can be performed in parallel. The new MPC-problem is formed by introducing the common variables in the extended problem~\eqref{eq:org_eqc_problem_expanded} and finding the parametric solutions to the subproblems~\eqref{eq:mp_subproblem} and~\eqref{eq:mp_subproblem_last}.

Theorem~\ref{thm:reduce_mpc} states that the original problem $\MPC{N}$ can be solved by first reducing it to $\MPC{p}$ with $p<N$, and then solve the smaller $\MPC{p}$ to determine the optimal parameters of the subproblems. However, $\MPC{p}$ can instead be reduced again to obtain an even smaller MPC problem, and in this section Theorem~\ref{thm:reduce_mpc} will be used repeatedly to obtain a problem structure that can be solved in parallel. This can be summarized in a tree structure, see Fig.~\ref{fig:arb_tree_struct}. Let the MPC problem at level $k$ be denoted $\MPC{p_{k-1}}$, and let $\xh{i}^{k-1}$ and $\uh{i}^{k-1}$ be the corresponding decision variables. Furthermore, let $\MPCsub{i}{k}$ be the $i$:th subproblem with prediction horizon $N_i^k$ at level $k$. The problem $\MPC{p_{k-1}}$ is reduced to the equivalent $\MPC{p_{k}}$ by solving all subproblems $\MPCsub{i}{k}, \; i \in \intset{0}{p_k}$ parametrically according to Section~\ref{sec:time_split}. Since all subproblems $\MPCsub{i}{k}$ are independent, this can be done in parallel.
The reduction of the MPC problem is continued until a problem with the minimal desired prediction horizon $p_{m-1} = N_0^{m}$ is obtained.
\begin{figure}
\centering
\def\svgwidth{\columnwidth}
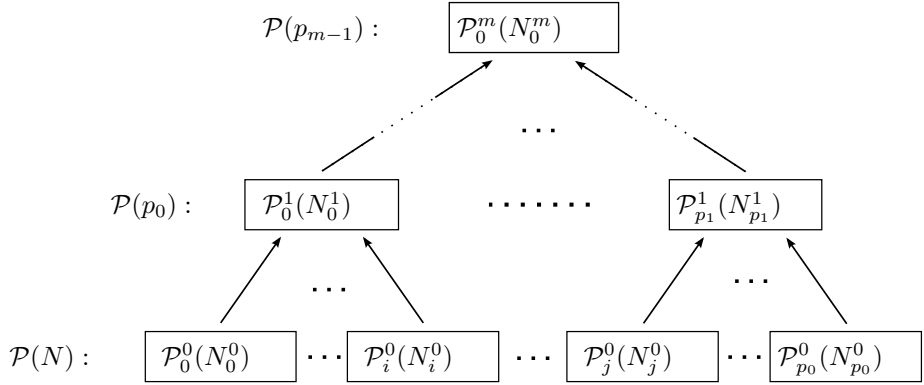
\caption{The tree structure that arises when the MPC problems are reduced in several steps. Each level in the tree forms an MPC problem that is again split into several smaller problems. The rectangles represents the subproblems, and the dotted lines represents that the structure is repeated. Here $0 < i < j < p$ are indices of subproblems, $m+1$ is the number of levels in the tree and $p_{m-1} = N_0^m$ is the minimal prediction horizon.}
\label{fig:arb_tree_struct}
\end{figure}

The original problem $\MPC{N}$ is solved by first reducing the problem in $m$ steps until $\MPC{p_{m-1}}$ is obtained, and then propagating the solution of $\MPC{p_{{m-1}}}$ down to level $k=0$. For subproblems $\MPCsub{i}{k}$ that are non-degenerate, the optimal primal and dual solutions are uniquely determined by the parameters $\xh{i}^{k}$ and $\uh{i}^{k}$ computed by their parents. For the primal degenerate subproblems, the dual solution has to be chosen according to~\eqref{eq:mp_dual_sol_corrected} and is also dependent on $\lamH{i}$. Since information is exchanged between parents and children only, the optimal solution to each $\MPCsub{i}{k}$ can be computed individually from the other subproblems at level~$k$. Hence, this can be performed in parallel.
\begin{remark}
Note that at each level $k$ in the tree in Fig.~\ref{fig:arb_tree_struct}, the common variables for level $k-1$ are computed. Hence, the consensus step to decide the common variables are done in one iteration and it is not necessary to iterate to get consensus between the subproblems as in many other methods.
\end{remark}
%\begin{thm} \label{thm:parallel_mpc_red}
%Consider an MPC problem $\MPC{N}$ on the form~\eqref{eq:org_eqc_problem} and let Assumption~\ref{assum:lin_indep} hold. Then $\MPC{N}$ can be solved in parallel by first solving all subproblems $\MPCsub{i}{k}$ in Fig.~\ref{fig:arb_tree_struct} parametrically, and then propagating the optimal solution of $\MPC{p_m}$ to the bottom level.
%\end{thm}
%\begin{pf}
%%For proof of Theorem~\ref{thm:parallel_mpc_red}, see Appendix~\ref{app:subsec:proof_thm_parallel_mpc}.
%The theorem follows directly by applying Theorem~\ref{thm:reduce_mpc} repeatedly and using the theory presented in Section~\ref{sec:time_split} and~\ref{sec_reduced_mpc}. Since all subproblems $\MPCsub{i}{k}$ at level $k$ can be solved individually, this can be performed in parallel. 
%\QED
%\end{pf}
%\begin{remark}
%If all (or some) subproblems $\MPCsub{i}{k}$ are the same, then only one solution has to be computed for these. This can lower the computational effort significantly for certain instances.
%\end{remark}

%%%%%%%%%%%%%%%%%%%%%%%%%%%%%%%%%%%%%%%%%%%%%%%%%%%%%%%%%%%%%%%%%%%%%%%
%%%%%%%%%%%%      PARALLEL ALGORITHM   %%%%%%%%%%%%%%%%%%%%%%%%%%%%%%%
%%%%%%%%%%%%%%%%%%%%%%%%%%%%%%%%%%%%%%%%%%%%%%%%%%%%%%%%%%%%%%%%%%%%%%%

\section{Parallel computation of Newton step}
The theory presented in this paper is summarized in Algorithms~\ref{alg:build_tree} and~\ref{alg:propagate_solution}. The algorithms can be used to compute the Newton step which is defined by the solution to~\eqref{eq:org_eqc_problem}. This is where most computational effort is needed when solving~\eqref{eq:min_problem}. The computations can be performed using several processors, and the level of parallelism can be tuned to fit the hardware, \ie the number of processing units, memory capacity, bus speed and more. The level of parallelism is decided by adjusting the number of subproblems at each level in the tree in Fig.~\ref{fig:arb_tree_struct}.

\subsection{Algorithms for parallel Newton step computation}
According to Section~\ref{sec:reduced_mpc} the algorithm for solving $\MPC{N}$ in parallel is based on two major steps; solve the subproblems $\MPCsub{i}{k}$ parametrically and propagate the solution downwards level for level. In both steps standard parallel numerical linear algebra could be used to parallelize further, \eg matrix multiplications, backward and forward substitutions and factorizations. This paper focuses on parallelization using the inherent structure of the MPC problem, and the discussion about possibilities to parallelize the computations will be limited to this scope.
%\begin{algorithm}
%  \caption{Parallel MPC-solver} \label{alg:parallel_mpc}
%  \begin{algorithmic}[1]
%  \STATE Reduce MPC-problem and create the tree structure using Algorithm~\ref{alg:build_tree}
%  \STATE Propagate solution down in the tree using Algorithm~\ref{alg:propagate_solution}
%  \end{algorithmic}
%\end{algorithm}

The first step, to construct the tree in Fig.~\ref{fig:arb_tree_struct}, is summarized in Algorithm~\ref{alg:build_tree}. Since all subproblems are independent of each other, the parfor-loop on Line~\ref{alg:line:build_tree_loop} to~\ref{alg:line:build_tree_loop_end} in Algorithm~\ref{alg:build_tree} can be performed in parallel on different processors. Let $p_{\max}$ be the maximum number of subproblems at any level in the tree. Then, if there are $p_{\max}$ processors available, all subproblems $\MPCsub{i}{k}$ at level $k$ can be solved simultaneously. At Line~\ref{alg:line:mp_solve} any suitable method could be used to find the matrices in the affine expressions of the optimal solutions to the subproblems.
\begin{algorithm}[h!]
  \caption{Parallel reduction of MPC problem} \label{alg:build_tree}
  \begin{algorithmic}[1]
  	\STATE Initiate level counter $k:=0$
  	\STATE Initiate the first number of subsystems $p_{-1} = N$
  	\STATE Set the minimal number of subproblems $p_{\min}$
    \WHILE{$p_k > p_{\min}$}
    	\STATE Compute desired $p_k$ to define the number of subproblems (with $p_k < p_{k-1}$)
    	\STATE Split the prediction horizon $0,\ldots, p_{k-1}$ in $p_k+1$ segments
        $0,\ldots, N_0^k$ up to $0,\ldots, N_{p_k}^k$
        \STATE Create subproblems $i=0,\ldots,p_k$ for each time block
   	    \STATE \textbf{parfor} $i = 0,\ldots,p_k$ \textbf{do} \label{alg:line:build_tree_loop}
        	\STATE \hspace{2ex} Solve subproblem $i$ parametrically and store $\Kx{i}$, \\ \hspace{2ex} $\kx{i}$, $\Kl{i}$ and $\kl{i}$ \label{alg:line:mp_solve}
      		\STATE \hspace{2ex} Compute $\Ah{i}$, $\Bh{i}$, $\ah{i}$, $\Hbh{i}$, $\fbh{i}$ and $\cbh{i}$ \\ \hspace{2ex} for the next level
      		\STATE \hspace{2ex} Compute and store $Z_i$
     	\STATE \textbf{end parfor} \label{alg:line:build_tree_loop_end}
     	\STATE Update level counter $k:= k + 1$
    \ENDWHILE
    \STATE Compute maximum level number $k:= k-1$ 
  \end{algorithmic}
\end{algorithm}

The second step is to propagate the solution down in the tree until the bottom level is reached. This is summarized in Algorithm~\ref{alg:propagate_solution}. Since all subproblems in the tree only use information from their parents, the parfor-loop at Line~\ref{alg:line:propagate_solution_loop} to Line~\ref{alg:line:propagate_solution_loop_end} can be computed in parallel. As for the first step, if there is one processor for each subproblem, all problems at each level in the tree can be solved simultaneously. 
\begin{algorithm}
  \caption{Parallel propagation of solution} \label{alg:propagate_solution}
  \begin{algorithmic}[1]
  	\STATE Initialize the first parameter as $\bar x$
  	\STATE Get level counter $k$ from Algorithm~\ref{alg:build_tree}
    \WHILE{$k \geq 0$}
	    \STATE \textbf{parfor} {$i=0,\ldots,p_k$} \textbf{do}\label{alg:line:propagate_solution_loop}
    		\STATE \hspace{2ex} Compute primal solution given by~\eqref{eq:mp_primal_sol} or~\eqref{eq:mp_last_primal_sol}
    		\STATE \hspace{2ex}  Compute dual solution given by~\eqref{eq:mp_dual_sol} or~\eqref{eq:mp_last_dual_sol}
    		\STATE \hspace{2ex}  \textbf{if} {Primal degenerate subproblem}
    			 \STATE \hspace{4ex} Select the dual solution according to~\eqref{eq:mp_dual_sol_corrected}
    		\STATE \hspace{2ex} \textbf{end if}
    	\STATE \textbf{end parfor} \label{alg:line:propagate_solution_loop_end}
    	\IF{k==0}
    		\STATE Compute $\nu_i$ according to Algorithm~\ref{alg:ineq_dual}
    	\ENDIF
   		\STATE Update level counter $k:=k-1$
    \ENDWHILE
  \end{algorithmic}
\end{algorithm}

The equality constrained problem~\eqref{eq:org_eqc_problem} was formed by eliminating the inequality constraints in~\eqref{eq:min_problem} that hold with equality. The dual variables $\nu$ corresponding to these eliminated constraints are important in \eg AS methods and can be computed as 
\begin{equation}
\nu_{i,t} = H_{xv,t,i}^T x_{t,i} + H_{uv,t,i}^Tu_{t,i} + B_{v,t,i}^T \lamS{t}{i} + f_{v,t,i} + H_{v,t,i}^T v_{t,i}, \label{eq:ineq_dual_comp}
\end{equation}
for $t \in \intset{0}{N_i-1}$ for each subproblem $i=0,\ldots,p$.  Here $v_{t,i}$ are the values of the eliminated control signals in~\eqref{eq:min_problem}. For the derivation of this expression, see \eg~\cite{axehill08:thesis}. The computation of $\nu$ is described in Algorithm~\ref{alg:ineq_dual}, which can be performed in parallel if $p_{\max}$ processors are available. Note that each $\nu_{t,i}$ in each subproblem can be computed in parallel if even more processors are available.
\begin{algorithm}
	\caption{Compute eliminated dual variables} \label{alg:ineq_dual}
	\begin{algorithmic}[1]
		\STATE \textbf{parfor} $i=0,\ldots,p_0$ \textbf{do}
		\STATE \hspace{2ex} \textbf{parfor} {$t=0,\ldots,N_i-1$} \textbf{do}
		\STATE \hspace{4ex} Compute $\nu_i$ according to~\eqref{eq:ineq_dual_comp}.
		\STATE \hspace{2ex} \textbf{end parfor}
		\STATE \textbf{end parfor}
	\end{algorithmic}
\end{algorithm}

So far no assumptions on the length of the prediction horizon of each subproblem has been made. If however the lengths of each subsystem is fixed to $s$, and the prediction horizon of the original problem is chosen as $N=s^{m+1}$ for simplicity, then the tree will get ${m+1}$ levels. Furthermore, assume that $s^{m}$ processors are available. Then, using the method proposed in~\cite{TondellMPQP} at Line~\ref{alg:line:mp_solve} in Algorithm~\ref{alg:build_tree}, each level in the tree is solved in roughly $\Ordo{n_{x}^3+\bar n_{u}^3}$ complexity (where $\bar n_u$ is the maximum control signal dimension at any level). Hence, the complete solution is obtained in roughly $\Ordo{m(n_{x}^3+\bar n_{u}^3)}$ complexity. Since ${m = \textrm{log}_s(N)-1}$ the computational complexity grows logarithmically in the prediction horizon, \ie as $\Ordo{\log N}$. 

The optimal length $s$ of the subproblems could be adjusted to fit the hardware which the algorithms are implemented on. Depending on the number of processors, the available memory and the communication delays between processors, the size of $s$ might be adjusted. The choice $s=2$ corresponds to a binary tree structure in Fig.~\ref{fig:arb_tree_struct}, and if the communication delays are negligible and there are sufficiently many processors available, it can be expected that this will give the best possible performance.

\subsection{Numerical results}
\label{sec:num_res}
The proposed algorithm for computing the Newton step using Algorithm~\ref{alg:build_tree} and~\ref{alg:propagate_solution} has been implemented in \textsc{Matlab} and used to solve random stable MPC problems in the form~\eqref{eq:org_eqc_problem}. The algorithm has been implemented serially, and the parallel computation times are simulated by summing over the maximum solution time at each level in the tree. Hence, memory and communication delays have not been addressed but are assumed small in comparison to the cost of the computations. In the implemented algorithm the subproblems are solved and $\Kx{i}$, $\kx{i}$, $\Kl{i}$ and $\kl{i}$ are computed using the methods proposed in~\cite{TondellMPQP}. Note that any choice of method that computes these matrices could be used. The numerical results for the algorithm when solving Newton steps for problems with $n_x=15$, $n_u=10$ and $s=2$ are seen in Fig.~\ref{fig:complexity}. The computation times are averaged over several runs. Here, the proposed algorithm has been compared to a well known state-of-the-art serial algorithm based on the Riccati factorization from \eg~\cite{axehill08:thesis} which is known to have $\Ordo{N}$ complexity growth. From the figure, the linear complexity of the Riccati based algorithm is evident. It is not obvious from this plot that the complexity grows logarithmically for this implementation of the proposed parallel algorithm. However, it can be observed that the computational time required by the parallel algorithm is significantly less and the growth of the computational complexity is much lower.
\begin{figure}
\centering
\includegraphics[width=0.85\columnwidth]{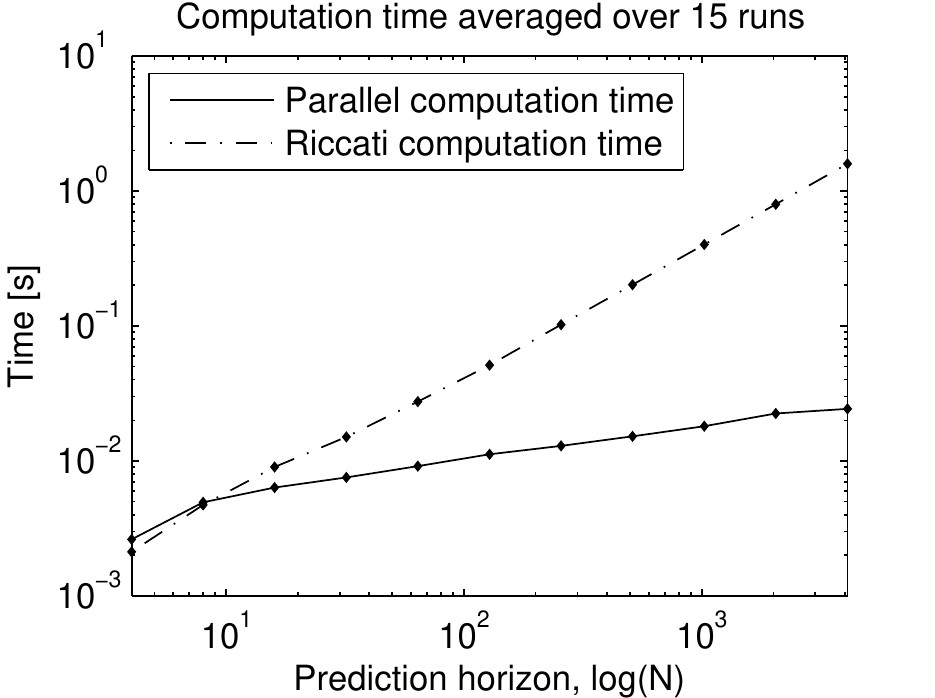}
\caption{Averaged solution times for the parallel solution of $\MPC{N}$ implemented in \textsc{Matlab}. It is compared to a serial algorithm based on Riccati factorization with $\Ordo{N}$ complexity. Here the systems are of the size $n_x=15$ and $n_u=10$, and $s=2$.}
\label{fig:complexity}
\end{figure}

The simulations were performed on an 
Intel Core i7-3517U CPU @ 1.9GHz running Windows 7 (version 6.1, build 7601: Service Pack 1) and \textsc{Matlab} (8.0.0.783, R2012b).

%%%%%%%%%%%%%%%%%%%%%%%%%%%%%%%%%%%%%%%%%%%%%%%%%%%%%%%%%%%%%%%%%%%%%%%%%%
%%%%%%%%%%%%%%%%%%%%%     CONCLUSION      %%%%%%%%%%%%%%%%%%%%%%%%%%%%%%%%%
%%%%%%%%%%%%%%%%%%%%%%%%%%%%%%%%%%%%%%%%%%%%%%%%%%%%%%%%%%%%%%%%%%%%%%%%%%%

\section{Conclusions}
\label{sec:conclusions}
In this paper a new algorithmic framework for computing Newton steps for MPC problems in parallel has been presented. It has been shown that the corresponding equality constrained MPC problem can be reduced to a new problem on the same form but with shorter prediction horizon in parallel. By repeating this in several steps, a tree structure of small MPC problems with short prediction horizons is obtained and can efficiently be solved in parallel. The proposed algorithm solves the Newton step arising in MPC problems in $\Ordo{\log (N)}$, \ie computational effort grows logarithmically in the prediction horizon $N$. In numerical experiments it has been shown that the proposed parallel algorithm outperforms an existing well known state-of-the-art serial algorithm. For future work, MPC problems with general linear constraints will be addressed and if the stability assumption can be removed if for example a pre-stabilization technique is employed.

%%%%%%%%%%%%%%%%%%%%%%%%%%%%%%%%%%%%%%%%%%%%%%%%%%%%%%%%%%%%%%%%%%%%%%%%%%
%%%%%%%%%%%%%%%%%%%%%     APPENDIX     %%%%%%%%%%%%%%%%%%%%%%%%%%%%%%%%%%%
%%%%%%%%%%%%%%%%%%%%%%%%%%%%%%%%%%%%%%%%%%%%%%%%%%%%%%%%%%%%%%%%%%%%%%%%%%

\newpage
\appendix
\section{Proofs}
\label{app:dual_proof}
The original equality constrained MPC problem is given by~\eqref{eq:org_eqc_problem}, where
\begin{equation}
f_t = \begin{bmatrix}
f_{x,t} \\ f_{u,t}
\end{bmatrix},
\end{equation}
and $\lambda_{t+1}$ is the dual variable corresponding to the equality constraint $x_{t+1} = A_tx_t+B_tu_t + a_t$.
Then the KKT system gives the following equations for $t\in \intset{0}{N-1}$
\begin{align}
&H_{t,x}x_t + H_{t,xu}u_t + f_{t,x} - \lambda_t + A_t^T \lambda_{t+1} = 0 \label{eq:app:kkt_x}\\
&H_{t,xu}^T x_t + H_{t,u }u_t + f_{t,u} + B_t^T \lambda_{t+1} = 0 \label{eq:app:kkt_u}\\
&x_{t+1} = A_t x_t + B_t u_t + a_t \label{eq:app:dyn}
\end{align}
and 
\begin{align}
&H_N x_N + f_N - \lambda_N = 0 \label{eq:app:kkt_last}\\
&x_0 = \bar x.
\end{align}

The extended problem that is composed of $p+1$ subproblems that share the common variables is given by~\eqref{eq:org_eqc_problem_expanded}. The common variables $\xbar{i}$ and $\dbar{i}$ are introduced as optimization variables in the extended problem. Let the dual variables for the subproblems $i=0,\ldots,p$ be defined by~\eqref{eq:lam0_dual_vars_ctrl}-\eqref{eq:lamd_dual_vars_ctrl},~\eqref{eq:lamh_min_one_ctrl} and~\eqref{eq:lamh_dual_vars_ctrl}. 
%\begin{align}
%\lamS{0}{i} \leftrightarrow x_{0,i} &= \bar x_i  \\
%\lamS{t+1}{i} \leftrightarrow x_{t+1,i}  &= A_{t,i}x_{t,i} + B_{t,i} u_{t,i} + a_{t,i}, \; t \in \intset{0}{N_i-1},
%\end{align}
%and for $i=0,\ldots,p-1$
%\begin{align}
%&\lamS{d}{i} \leftrightarrow x_{N_i,i} =  \mathcal{A}_i \bar x_i + \mathcal{T}_i \bar d_i + \mathcal{D}_i \mathbf{a}_i   \\
%&\lamH{i} \leftrightarrow \bar x_{i+1} =  \mathcal{A}_i \bar x_i + \mathcal{T}_i \bar d_i + \mathcal{D}_i \mathbf{a}_i  \\
%&\lamH{-1} \leftrightarrow \bar x_0 = \bar x .
%\end{align}
%The symbol $\leftrightarrow$ means that the dual variable and the equality constraint corresponds to each other.
Then the corresponding KKT system of this extended problem consists of the following equations (for all subproblems $i=0,\ldots,p$)
\begin{align}
&H_{x,t,i}x_{t,i} + H_{xu,t,i}u_{t,i} + f_{x,t,i} - \lamS{t}{i} + A_{t,i}^T \lamS{t+1}{i} = 0 \label{eq:app:kkt_ext_x}\\
&H_{xu,t,i}^T x_{t,i} + H_{u,t,i} u_{t,i} + f_{u,t,i} + B_{t,i}^T \lamS{t+1}{i} = 0 \label{eq:app:kkt_ext_u}
\end{align}
for $t \in \intset{0}{N_i-1}$. For the last subproblem there is also an equation corresponding to the last term in the objective function
\begin{equation}
H_{N_p,p}x_{N_p,p} + f_{N_p,p} - \lamS{N_p}{p} = 0. \label{eq:app:kkt_ext_last}
\end{equation}
Furthermore, the relation between the dual variables $\lamS{N_i}{i}$, $\lamS{0}{i}$, $\lamtc{i}$ and $\lamH{i}$ for $i=0,\ldots,p-1$ are given directly by the KKT system
\begin{align}
&\lamS{0}{p} = \lamH{p-1} \label{eq:app:lam0p}\\
& \lamS{0}{i} = \lamH{i-1} -  \Ac{i}^T(\lamtc{i} + \lamH{i}), \; t \in \intset{0}{p-1} \label{eq:app:lam0_min_lamh_plus_A_mult_lamd_plus_lamh}\\
&\Tc{i}^T(\lamtc{i} + \lamH{i}) = 0, \; t \in \intset{0}{p-1} \label{eq:app:T_mult_lamd_plus_lamh}\\
&\lamS{N_i}{i} = - \lamtc{i}, \; t \in \intset{0}{p-1} \label{eq:app:lamNi_eq_min_lamd}.
\end{align}
The primal feasibility constraints that must be satisfied in the KKT system are given by
\begin{align}
\forall \; i \in &\intset{0}{p} \begin{cases}
x_{0,i} = \xbar{i} \\
x_{t+1,i} = A_{t,i}x_{t,i} + B_{t,i}u_{t,i}+  a_{t,i}, \; t \in \intset{0}{N_{i}-1} \\
x_{N_{i},i} = d_{i}= \Ac{i} \xbar{i} + \Tc{i} \dbar{i} +  \ac{i}, \; i \neq p \end{cases}\\
\xbar{0} &= \bar x \\
%x_{,p} &= \xbar{p} \\
%x_{t+1,p} &= A_{t,p}x_{t,p}+B_{t,p}u_{t,p} + a_{t,p}, \; t \in \intset{0}{N_{p}-1}\\
\xbar{i+1} &= d_i = \Ac{i} \xbar{i} + \Tc{i} \dbar{i} + \ac{i}, \; i \in \intset{0}{p-1} \label{eq:app:coupling_cnstr}
\end{align}

%%%%%%%%%%%%%%%%      THEOREM REDUCE THE MPC PROBLEM    %%%%%%%%%%%%%%%%%%

\subsection{Proof of Theorem~\ref{thm:reduce_mpc}}
\label{app:subsec:proof_lemma_reduction}
The reduction of $\MPC{N}$ to $\MPC{p}$ with $p<N$ follows directly from the theory presented in Section~\ref{sec:time_split}.

The optimal primal variables in subproblem $i$ and $i+1$ are related as $x_{0,i+1}^* = x_{N_i,i}^*$, whereas the dual variables given by \eqref{eq:mp_dual_sol_corrected} are related according to~\eqref{eq:dual_equality_1}-\eqref{eq:dual_equality_2}. By inserting~\eqref{eq:dual_equality_1}-\eqref{eq:dual_equality_2} into~\eqref{eq:app:kkt_ext_x} and~\eqref{eq:app:kkt_ext_u} and using $x_{0,i+1}^* = x_{N_i,i}^*$, the resulting equations are identical to~\eqref{eq:app:kkt_x} and~\eqref{eq:app:kkt_u}. Hence, the solution to the system of equations defined by~\eqref{eq:dual_equality_1}-\eqref{eq:dual_equality_2} and~\eqref{eq:app:kkt_ext_x}-\eqref{eq:app:kkt_ext_last} is a solution to the original KKT system of the problem in~\eqref{eq:org_eqc_problem}. Assumption~\ref{assum:lin_indep} gives uniqueness of the solution and the unique optimal solution to~\eqref{eq:org_eqc_problem} can hence be obtained as
\begin{equation}
X^* = \begin{bmatrix}
x_0^* \\ \vdots \\ x_{N_0}^* \\ u_{N_0}^* \\ \vdots \\ x_N^*
\end{bmatrix} = \begin{bmatrix}
x_{0,0}^* \\ \vdots \\ x_{N_0,0}^* \\ u_{0,1}^* \\ \vdots \\ x_{N_p,p}^*
\end{bmatrix}, \; \lambda^* = 
\begin{bmatrix}
\lambda_0^* \\ \vdots \\ \lambda_{N_i}^* \\ \lambda_{N_i+1}^* \\ \vdots \\ \lambda_N^*
\end{bmatrix} = \begin{bmatrix}
\lamS{0}{0}^* \\ \vdots \\ \lamS{N_i}{0}^* \\ \lamS{1}{1}^* \\ \vdots \\ \lamS{N_p}{p}^*
\end{bmatrix}.
\end{equation}
\QED

%%%%%%%%%%%%%%%       THEOREM NULLSPACE     %%%%%%%%%%%%%%%%%%%%%%%%%%

\subsection{Proof of Theorem~\ref{thm:nullspace}}
\label{app:nullspace}
The null space of $\Ae{i}^T$ is given by all $\lamN{i}$ such that $\Ae{i}^T \lamN{i} = 0$, which can be expressed as
\begin{align}
&-\lamN{t,i} + A_{t,i}^T \lamN{t+1,i} = 0, \; t \in \intset{0}{N_i-1} \label{pf:thm:nullspace:lamATlam}\\
&B_{t,i}^T \lamN{t+1,i} = 0, \; t \in \intset{0}{N_i-1} \label{pf:thm:nullspace:BTlam}\\
&\lamN{N_i,i} = - \lamN{tc,i} . \label{pf:thm:nullspace:lamNi}
\end{align}
Equation~\eqref{pf:thm:nullspace:lamATlam} and~\eqref{pf:thm:nullspace:lamNi} can be combined into
\begin{align}
\lamN{i} = \begin{bmatrix}
\lamN{0,i} \\ \vdots \\ \lamN{tc,i}
\end{bmatrix}=\begin{bmatrix}  
- \Ac{i}^T \\-\mathcal{D}_i^T \\ I
\end{bmatrix} \lamN{tc,i},
\end{align}
where $\Ac{i}$ and $\mathcal{D}_i$ are defined as in~\eqref{eq:d_range} and~\eqref{eq:definiton_A_S_D}. By using~\eqref{pf:thm:nullspace:BTlam}, $\lamN{tc,i}$ has to satisfy $\lamN{tc,i} \in \ker{\mathcal{S}_i^T}$. For notational convenience, let $w_i = \lamN{tc,i}$ and define $\Nb{i}$ as
\begin{equation}
\Nb{i} \triangleq \begin{bmatrix}
-\Ac{i} & -\mathcal{D}_i & I  
\end{bmatrix}^T.
\end{equation} 
Then the null space element $\lamN{i}$ is computed as
\begin{equation}
\lamN{i} = Z_i w_i, \quad w_i \in \ker{\mathcal{S}_i^T}.
\end{equation} 
\QED

%%%%%%%%%%%%%%%       THEOREM (PRIMAL DEGENERATE)   %%%%%%%%%%%%%%%%%%%%%%%%%%%%%%%
\subsection{Proof of Theorem~\ref{thm:dual_vars_overdet}}
\label{app:subsec:proof_thm_overdet}
The equations~\eqref{eq:app:kkt_ext_x}-\eqref{eq:app:coupling_cnstr} are given by KKT system of the extended MPC problem~\eqref{eq:org_eqc_problem_expanded} that consists of $p+1$ subproblems . The relations between the optimal dual variables in different subproblems are directly given by~\eqref{eq:app:lam0p}-\eqref{eq:app:lamNi_eq_min_lamd}.
\QED

%%%%%%%%%%%%%%%%%%%      THEOREM PARALLEL %%%%%%%%%%%%%%%%%
%
%\subsection{Proof of Theorem~\ref{thm:parallel_mpc_red}}
%\label{app:subsec:proof_thm_parallel_mpc}
%Consider an MPC problem $\MPC{N}$ and let Assumption~\ref{assum:lin_indep} hold. Theorem~\ref{thm:reduce_mpc} states that $\MPC{N}$ can be reduced to a smaller problem on the same form, but with prediction horizon $1 \leq p < N$. The reduction is performed by splitting $\MPC{N}$ into $p+1$ subproblems $\MPCsub{i}{}$ and solving the subproblems parametrically. This reduction can be performed in several steps to obtain a problem with the desired prediction horizon.

\bibliographystyle{plain}
%\bibliographystyle{alpha}        % Include this if you use bibtex 
%\bibliography{autosam}           % and a bib file to produce the 
\bibliography{IEEEfull,axe_full,ianFull}
                                 % bibliography (preferred). The
                                 % correct style is generated by
                                 % Elsevier at the time of printing.

\end{document}